\documentclass[a4paper,10pt]{article}
\usepackage[utf8]{inputenc}
\usepackage{amsmath,amsthm, amssymb}
\usepackage{amsfonts, amsbsy}
\usepackage[mathscr]{euscript}
\usepackage{xspace}
\usepackage{graphicx,subfigure}
\usepackage{color}
\usepackage{enumerate}
\usepackage[english]{babel}
\usepackage{cite}
\usepackage{verbatim}
\usepackage{tikz}
\usepackage{fullpage}
\usepackage{authblk}

\newcommand{\N}{\mathbb{N}}

\theoremstyle{definition}
\newtheorem{theorem}{Theorem}
\newtheorem{proposition}[theorem]{Proposition}
\newtheorem{corollary}[theorem]{Corollary}
\newtheorem{lemma}[theorem]{Lemma}
\newtheorem{claim}{Claim}[theorem]

\theoremstyle{definition}
\newtheorem{problem}[theorem]{Problem}
\newtheorem{conjecture}[theorem]{Conjecture}

\newcommand{\ra}{\rightarrow}

\newcommand{\dic}{\vec{\chi}}

\DeclareMathOperator{\mad}{mader}
\DeclareMathOperator{\dist}{dist}

\DeclareMathOperator{\cc}{cc}

\newenvironment{subproof}{\par\noindent {\it Subproof}.\ }{\hfill$\lozenge$\par\vspace{11pt}}

\graphicspath{{./figures/}}

\title{Subdivisions  in digraphs of large out-degree or large dichromatic number\thanks{This work was supported by ANR under contract STINT ANR-13-BS02-0007}}

\author[1]{Pierre Aboulker}
\author[2]{Nathann Cohen}
\author[1]{Fr\'ed\'eric Havet}
\author[1,3]{William Lochet}
\author[4,3]{Phablo~F.~S.~Moura\thanks{Supported by FAPESP Proc.~2013/19179-0, 2015/11930-4}}
\author[3]{ St\'ephan Thomass\'e}
\affil[1]{Université Côte d’Azur, CNRS, Inria, I3S, France}
\affil[2]{CNRS, Université Paris Sud, France}

\affil[3]{LIP, ENS de Lyon, CNRS, Université de Lyon, France}
\affil[4]{Instituto de Matemática e Estatística, Universidade de São Paulo, Brazil}

\date{\today}

\begin{document}

\maketitle

\begin{abstract}
%

In 1985, Mader conjectured the existence of a function~$f$ such that every digraph with minimum out-degree at least $f(k)$ contains a subdivision of the transitive tournament of order $k$. 
This conjecture is still completely open, as the existence of $f(5)$ remains unknown. 
In this paper, we show that if~$D$ is an oriented path, or an in-arborescence (i.e., a tree with all edges oriented towards the root)
or the union of two directed paths from $x$ to $y$ and a directed path from $y$ to $x$, then every digraph with minimum out-degree large enough contains a subdivision of~$D$.
Additionally, we study Mader's conjecture considering another graph parameter.
The \emph{dichromatic number} of a digraph $D$ is the smallest integer $k$ such that $D$ can be partitioned into $k$ acyclic subdigraphs. 
We show that any digraph with dichromatic number greater than $4^m (n-1)$ contains every digraph with $n$ vertices  and $m$ arcs  as a subdivision. 
\end{abstract}

\section{Introduction}
Mader~\cite{Mad67} established the following.

\begin{theorem} [Mader~\cite{Mad67}]\label{thm:mader-undirected}
There exists an integer~$g(k)$ such that every graph with minimum degree at least $g(k)$ contains a subdivision of $K_k$. 
\end{theorem}
 For $k\leq 4$, we have $g(k)=k-1$ as first proved by Dirac~\cite{Dirac60}; for $k=5$, we have the estimate $6\leq g(5)\leq 7$ by Thomassen~\cite{Thomassen74,Thomassen96}. In general, the order of growth of $g(k)$ is $k^2$ as shown in
\cite{BoTh98} and \cite{KoSz96}.

Similarly, it would be interesting to find analogous results for digraphs.
However, the obvious analogue that a digraph with sufficiently large minimum in- and out-degree contains a subdivision of the complete digraph of order $n$ is
false as shown by Mader~\cite{Mader85}.

Let $\gamma$ be a digraph parameter. 
A digraph $F$ is $\gamma$-maderian if there exists a least integer
$\mad_{\gamma}(F)$ such that every digraph $D$ with $\gamma(D)\geq \mad_{\gamma}(F)$ contains a subdivision of $F$. 

For a digraph $D$, $\delta^+(D)$ (resp.
$\delta^-(D)$) denote the minimum out-degree (resp.~in-degree) and $\delta^0(D)=\min\{\delta^+(D),\delta^-(D)\}$.
A natural question is to ask which digraphs $F$ are {\it $\delta^+$-maderian} (resp. {\it $\delta^0$-maderian}). 
Observe that every $\delta^+$-maderian digraph is also $\delta^0$-maderian and that
$\mad_{\delta^+}  \geq \mad_{\delta^0}$.

For each positive integer~$k$, we denote by~$[k]$ the subset of the natural numbers~$\{1, \ldots, k\}$.

\medskip

On the positive side, Mader conjectured that every acyclic digraphs is $\delta^+$-maderian.
Since every acyclic digraph is the subdigraph of the transitive tournament on the same order, it is enough to prove that transitive tournaments
are $\delta^+$-maderian.

\begin{conjecture}[Mader~\cite{Mader85}]\label{mader}
There exists a least integer $\mad_{\delta^+}(TT_k)$ such that every digraph $D$ with $\delta^+(D) \ge \mad_{\delta^+}(TT_k)$ contains a subdivision of $TT_k$.
\end{conjecture}
Mader proved that $\mad_{\delta^+}(TT_4) = 3$, but even the existence of $\mad_{\delta^+}(TT_5)$ is still open.
 
This conjecture implies directly that transitive tournaments (and thus all acyclic digraphs) are $\delta^0$-maderian.

\begin{conjecture}\label{mader-bis}
There exists a least integer $\mad_{\delta^0}(TT_k)$ such that every digraph $D$ with $\delta^0(D) \ge \mad_{\delta^0}(TT_k)$ contains a subdivision of $TT_k$.
\end{conjecture}

In fact, Conjecture~\ref{mader-bis} is equivalent to Conjecture~\ref{mader} because if transitive tournaments are $\delta^0$-maderian, 
then $\mad_{\delta^+}(TT_k)\leq \mad_{\delta^0}(TT_{2k})$ for all $k$. 
Indeed, let $D$ be a digraph with minimum out-degree $\mad_{\delta^0}(TT_{2k})$, and let $D'$ be the digraph obtained from disjoint copies 
of $D$ and its converse (the digraph obtained by reversing all arcs) $\overline{D}$  by adding all arcs from $\overline D$ to $D$.
Clearly, $\delta^0(D')\geq \mad_{\delta^0}(TT_{2k})$. Therefore $D'$ contains a subdivision of $TT_{2k}$.
Hence, either $D$ or $\overline{D}$ (and so $D$) contains a subdivision of~$TT_k$.

Both conjectures are equivalent, but the above reasoning does not prove that a $\delta^0$-maderian digraph is also $\delta^+$-maderian. 
The case of {\it oriented trees} (i.e. orientations of undirected trees) is typical. 
Using a simple greedy procedure, one can easily find every oriented tree of order~$k$ in every digraph with minimum in- and out-degree~$k$ 
(so $\mad_{\delta^0}(T)=|T|-1$  for any oriented tree $T$).
On the other hand, it is still open whether oriented trees are $\delta^+$-maderian and a natural important step towards 
Conjecture~\ref{mader} would be to prove the following weaker one.

\begin{conjecture}\label{conj:tree}
Every oriented tree is $\delta^+$-maderian.
\end{conjecture}
We give evidences to this conjecture.
First, in Subsection~\ref{subset:path}, we prove that every {\it oriented path} (i.e. orientation of an undirected path) $P$ is $\delta^+$-maderian and that $\mad_{\delta^+}(P)=|V(P)|-1$.
Next, in Subsection~\ref{subsec:inarb}, we  consider arborescences. An {\it out-arborescence} (resp. {\it in-arborescence}) is an  oriented tree in which all arcs are directed away from (resp. towards) a vertex called the {\it root}.
Trivially, the simple greedy procedure shows that $\mad_{\delta^+}(T)=|T|-1$ for every out-arborescence.
In contrast, the fact that in-arborecences are $\delta^+$-maderian is not obvious since we have no control on the in-degree of each vertex in a digraph of out-degree at least $k$. We show in Theorem~\ref{inTrees} that the in-arborescences are $\delta^+$-maderian.

In \cite{Mader95}, Mader gave another partial result towards Conjecture~\ref{mader}. He proved the existence of a function $f_1(k)$ such that every digraph $D$ with $\delta^+(D) \geq f_1(k)$ contains a pair
of vertices $u$ and $v$ with $k$ disjoint directed paths between $u$ and $v$. 
It is however not known if we can insist on 
these paths to be arbitrarily long. In Theorem \ref{C(k,k)}, we do the first step toward this question by proving it in the case $k = 2$.

\medskip

Conjecture~\ref{mader} states that the acyclic digraphs are $\delta^+$-maderian. However, they are not the only ones. For example, it is folklore that every digraph with minimum out-degree at least $1$ contains a directed cycle, which is a subdivision of $\vec{C}_2$, the directed $2$-cycle. More generally, one can easily show, by considering a directed path of maximum length, that every digraph
with minimum out-degree at least $k-1$ contains a directed cycle of length at least $k$. In other words, $\mad_{\delta^+}(\vec{C}_k) = k-1$, where $\vec{C}_k$ denotes the directed $k$-cycle.
Furthermore Alon~\cite{Alon96} showed that every digraph with minimum out-degree at least $64k$ contains $ k $ disjoint directed cycles, which forms a subdivision of the
disjoint union of $k$ copies of~$\vec{C}_2$.
A celebrated conjecture of Bermond and Thomassen~\cite{BeTh81} states that the bound $64k$ can be decreased to~$2k-1$.

\begin{conjecture}[Bermond and Thomassen~\cite{BeTh81}]\label{Bermond Thomassen}
For every positive integer $ k $, every digraph with minimum out-degree at least $ 2k-1 $ contains $ k $ disjoint directed cycles.
\end{conjecture}

In~\cite{Alon96}, Alon also conjectured the following.

\begin{conjecture}[Alon~\cite{Alon96}]
There exists a function $h$ such that every digraph with minimum out-degree $h(k)$ has a partition $(V_1,V_2)$ such that, for $i=1,2$, $D\langle V_i \rangle$ has minimum out-degree $k$.
\end{conjecture}

The difficulty of this question is remarkable, as the existence  of $h(2)$ still remains open. 
If true, this conjecture implies the following one.

\begin{conjecture}\label{conj:union}
If $F_1$ and $F_2$ are $\delta^+$-maderian, then the disjoint union of $F_1$ and~$F_2$ is also $\delta^+$-maderian.
\end{conjecture}

Partial positive answers to this conjecture can be obtained via the Erd\H{o}s-Pos\'a Property.
For a digraph $F$ and an integer $k$, we denote by $k\times F$ the disjoint union of $k$ copies of $F$.
A digraph $F$ is said to have the {\it Erd\H{o}s-Pos\'a Property} if for every positive integer $k$, there exists $\phi(k)$, such that every digraph $D$ either contains a subdivision of $k\times F$, or has a set $S$ of size (at most) $\phi(k)$ such that $D-S$ contains no subdivision of $F$.

\begin{theorem}\label{thm:erdos_posa_property}
If $F$ is $\delta^+$-maderian and has the Erd\H{o}s-Pos\'a Property, then $k\times F$ is also $\delta^+$-maderian for all positive integer $k$.
\end{theorem}
\begin{proof}
Let $F$ be a maderian digraph having the Erd\H{o}s-Pos\'a Property for some function $\phi$.
Let $D$ be a digraph with $\delta^+(D)\geq \phi(k) + \mad_{\delta^+}(F)$.
For every set~$S$ of size at most $\phi(k)$, $\delta^+(D-S)\geq \mad_{\delta^+}(F)$, so $D-S$ contains a subdivision of $F$. 
Thus, since $F$ has the  Erd\H{o}s-Pos\'a Property, $D$ contains a subdivision of~$k\times F$.
\end{proof}

Reed et al.~\cite{RRST96} proved that $\vec{C}_2$ has the Erd\"os-Pos\'a Property, and using the Directed Grid Theorem~\cite{KaKr15},  
Akhoondian et al.~\cite{AKKW16} showed that many digraphs have the Erd\"os-Pos\'a Property, 
in particular all  disjoint unions of directed cycles. 
Hence, Theorem~\ref{thm:erdos_posa_property} implies that disjoint unions of directed cycles are $\delta^+$-maderian.

\medskip

On the negative side,
Thomassen~\cite{Thomassen86} showed a construction of digraphs with arbitrarily large in- and out-degree and no directed cycles of even length (see also \cite{DMMS12}).
This gives a large class of digraphs that are  not $\delta^0$-maderian, namely: 
digraphs such that all its subdivisions have a directed cycle of even length. 
In particular, $\vec{K}_3$, the complete digraph on three vertices, belongs to this class. 
All  digraphs in this class have been characterized by Seymour and Thomassen~\cite{SeTh87}.
Morever, Devos~et~al.~\cite{DMMS12}
showed the existence of digraphs with large minimum out-degree and in-degree such that there exists no pair of vertices $u$ 
and $v$ with four internally disjoint directed paths between them, two from $u$ to $v$ and two from $v$ to $u$. 
In Theorem \ref{triplePath}, we show that if the minimum out-degree is large enough then we can find two vertices $u$ and
$v$ with three internally disjoint directed paths between them, one from $u$ to $v$ and two from $v$ to $u$.





Note that every graph with chromatic  number at least~$p$ has a subgraph with minimum degree at least~$p-1$. 
This implies, by Theorem~\ref{thm:mader-undirected}, that every graph with chromatic number at least $g(k)+1$ 
contains a subdivision of $K_k$. 
In the context of digraphs, there exist two natural analogues of the chromatic number.
Given a digraph~$D$, the \emph{chromatic number} of~$D$, denoted by~$\chi(D)$, is simply the chromatic number of its underlying graph. 
The \emph{dichromatic number} of~$D$, denoted by~$\dic (D)$, is the smallest integer~$k$ such that~$D$ admits a $k$-dicolouring. 
A \emph{$k$-dicolouring} is a $k$-partition $\{V_1, \dots , V_k\}$ of $V(D)$ such that~$D\langle V_i \rangle$ is acyclic 
for every $i \in [k]$.
Hence, it is natural to ask which digraphs are $\chi$-maderian and which ones are $\dic$-maderian.

Burr~\cite{Bur80} proved that every $(k-1)^2$-chromatic digraph contains every oriented forest of order $k$.
Later on, Addario-Berry et al.~\cite{AHL+13} slightly improved this value to~$k^2/2-k/2+1$.
This implies that every oriented forest is $\chi$-maderian.
Cohen et al.~\cite{CHLN} showed that for any positive integer $b$, there are digraphs of arbitrarily high chromatic number that contain no oriented cycles with less than $b$ blocks. This directly implies that if a digraph is not an oriented forest, then it is not $\chi$-maderian because it contains an oriented cycle, all subdivisions of which have the same number of blocks. 
\begin{theorem}
A digraph is $\chi$-maderian if and only if it is an oriented forest.
\end{theorem}

The $\chi$-maderian digraphs are known but determining $\mad_{\chi}$ for such digraphs is still open.
Burr~\cite{Bur80} made the following conjecture.

\begin{conjecture}[Burr~\cite{Bur80}]\label{conj:Burr}
Every digraph with chromatic number $2k-2$ contains every oriented tree of order $k$ as a subdigraph.
\end{conjecture}

An interesting  step towards Burr's conjecture is to prove the following consequence of it.

\begin{conjecture}
If $T$ is an oriented tree of order $k$, then $\mad_{\chi}(T)\leq 2k-2$.
\end{conjecture}

In Section~\ref{sec:dic}, we prove that every digraph is $\dic$-maderian.
Again determining $\mad_{\dic}$ for every digraph is still open. 
Since every digraph $D$ of order $n$ is a subdigraph of $\vec{K}_n$, the complete digraph of order $n$, and so $\mad_{\dic}(D) \leq \mad_{\dic}(\vec{K}_n)$, it is natural to focus on $\vec{K}_n$.
\begin{problem}
What is $\mad_{\dic}(\vec{K}_n)$ ?
\end{problem}

In Subsection~\ref{better-dic}, we show $\mad_{\dic}(\vec{K}_n) \leq 4^{n^2-2n+1}(n-1)+1$ and more generally that if $F$ is a digraph with $n$ vertices, $m$ arcs and $c$ connected components, then $\mad_{\dic}(F) \leq 4^{m-n+c}(n-1)+1$ (Corollary~\ref{cor:dic-Kn}).
We also give better upper bounds on $\mad_{\dic}$ for some particular digraphs.

\bigskip

To prove Theorem~\ref{thm:mader-undirected}, Mader showed a stronger result about the average degree. Recall that the {\it average degree} of a graph $G$ is $\bar{d}(G) = 2|E(G)|/|V(G)|$.
He proved that there exists a function $g'(k)$ such that every graph $G$ with at least $\bar{d}(G)\geq g'(k)$
contains a subdivision of $K_k$.  
The {\it average out-degree} of a digraph $D$ is   $\bar{d}^+(D)=|A(D)|/|V(D)|$.  (Note that this is equal to the average in-degree and half the average degree.)
A digraph is \emph{$\overrightarrow P_3$-free} if it does not contain $\overrightarrow P_3$ as a subdigraph, where $\overrightarrow P_3$ is the dipath on three vertices. 
There are bipartite graphs with arbitrarily large degree and arbitrarily large girth (recall that the {\it girth} of a graph is the length of a smallest cycle or $+\infty$ if it is a forest). 
Orienting edges of such graphs from one part to the other  result in  $\overrightarrow P_3$-free digraphs with arbitrarily large average out-degree and arbitrarily large girth (the girth of a digraph is the girth of its underlying graph). 
Consequently, a digraph is $\bar{d}^+$-maderian only  if it is an {\it antidirected forest}, 
that is, an oriented forest containing no $\overrightarrow P_3$ as a subdigraph. 
This simple necessary condition is  also sufficient.
Burr~\cite{Bur82} showed that all antidirected forests are $\bar d^+$-maderian: for every antidirected forest $F$, $\mad_{\bar d^+}(F)\leq 4|V(F)|-4$.

\begin{theorem}
A digraph is  $\bar d^+$-maderian if and only if it is an antidirected forest.
\end{theorem}

Addario-Berry et al.~\cite{AHL+13} conjectured that the bound $4|V(F)|-4$ in Burr's result is not tight.
\begin{conjecture}[Addario-Berry et al.~\cite{AHL+13}]\label{analog}
Let $D$ be a digraph.  If $|A(D)| > (k-2) |V(D)|$, 
then $D$ contains every antidirected tree of order $k$ as a subdigraph.
\end{conjecture}
The value $k-2$ in this conjecture would be best possible, since the oriented star $S^+_k$, consisting of a vertex dominating $k-1$ others, is not contained in any
digraph in which every vertex has out-degree $k-2$. It is also  
tight because the complete digraph $\vec{K}_{k-1}$ 
has $(k-2)(k-1)$ arcs but  trivially does not contain any oriented tree of order $k$.

As observed in~\cite{AHL+13}, Conjecture~\ref{analog} for oriented graphs implies Burr's conjecture (Conjecture~\ref{conj:Burr}) for antidirected trees and
Conjecture~\ref{analog} for symmetric digraphs is equivalent to the well-known Erd\H{o}s-S\'os
conjecture. 
\begin{conjecture}[Erd\H{o}s and S\'os~\cite{Erd65}]\label{erdossos}\rm 
Let $G$ be a graph.  If $|E(G)| > \frac{1}{2} (k-2) |V(G)|$, 
then $G$ contains every tree of order $k$ .
\end{conjecture}
Their conjecture  has attracted a fair amount of attention over  the last decades. 
Partial solutions are given in~\cite{BrDo96, Haxell2001, SaWo97}. 
In the early 1990's, Ajtai, Koml\'os, Simonovits and Szemer\' edi  announced a proof of this result for sufficiently large~$m$.

\bigskip

Since $k$-connected and $k$-edge-connected graphs have  minimum degree at least  $k$,  Theorem~\ref{thm:mader-undirected}  implies that every graph $G$ with connectivity (resp. edge-connectivity) at least $g(k)$ contains a subdivision of $K_k$.
Let $\kappa(D)$, $\kappa'(D)$, be respectively the strong connectivity and the strong arc-connectivity of $D$. 

\begin{problem}
Are all digraphs $\kappa$-maderian?  $\kappa'$-maderian ?
\end{problem}

The following conjecture due to Thomassen~\cite{Thomassen89} implies that all digraphs are $\kappa$-maderian.

\begin{conjecture}[Thomassen~\cite{Thomassen89}]
There exists $f_1(k)$ such that if $\kappa(D)\geq f_1(k)$,
and $x_1, \dots, x_k$ and $y_1, \dots , y_k$ are distinct vertices of $D$, then $D$ contains $k$ disjoint dipaths
$P_1, \dots , P_k$ such that $P_i$ goes from $x_i$ to $y_i$ for all $i \in [k]$. 
\end{conjecture}
This conjecture would also imply the following one due to Lov\'asz~\cite{Lovasz75}.
\begin{conjecture}
There exists a integer $p$ such that every $p$-strongly connected digraph has an even directed cycle.
\end{conjecture}

\section{Subdivision in digraphs with large minimum out-degree}\label{sec:tree}

\subsection{Subdivisions of oriented paths}\label{subset:path}

Let $P=(x_1,x_2,\cdots, x_n)$ be an oriented path.
We say that $P$ is an {\it $(x_1,x_n)$-path}.
The vertex $x_1$ is the {\it initial vertex} of $P$ and $x_n$ its {\it terminal vertex}.
$P$ is a {\it directed path} or simply a {\it dipath}, if $x_i\ra x_{i+1}$ for all~$i \in [n-1]$. 

Let~$k_1$ be nonnegative integer and $k_2, \ldots, k_\ell$ be positive integers.
  We denote by~$P(k_1, k_2, \ldots, k_{\ell})$ the path obtained from an undirected path
  $(v_1 v_2 \ldots v_{\ell+1})$ by replacing, for every~$i \in [\ell]$, the edge~$(v_i, v_{i+1})$ by a directed path
  of length~$k_i$ from~$v_i$ to $v_{i+1}$ if~$i$ is odd, and from $v_{i+1}$ to $v_{i}$ if~$i$ is even.
(If $k_1=0$, then $v_1=v_2$.)

\begin{theorem}\label{paths}
 Let~$P(k_1, k_2, \ldots, k_\ell)$ be a path, and let~$D$ be a digraph with~$\delta^+(D)\geq \sum_{i=1}^\ell k_i$.
 For every~$v \in V(D)$, $D$ contains a path $P(k^\prime_1, k^\prime_2, \ldots, k^\prime_\ell)$ with initial vertex~$v$ such that
 $k^\prime_i \geq k_i$ if~$i$ is odd, and $k^\prime_i = k_i$ otherwise.
\end{theorem}
\begin{proof}
 By induction on~$\ell$.
 If~$\ell = 1$, then the result holds trivially.
 Assume now~$\ell \geq 2$, and suppose that, for every path $P(x_1, x_2, \ldots, x_{t})$ with~$t < \ell$ and every
 digraph~$G$ with~$\delta^+ (G) \geq \sum_{i=1}^t x_i$,
 ~$G$ contains a path $P(x^\prime_1, x^\prime_2, \ldots, x^\prime_{t})$ starting at any vertex of~$G$ such that
 $x^\prime_i \geq x_i$ if~$i$ is odd, and $x^\prime_i = x_i$ otherwise.

 Let~$v$ be a vertex of~$D$.
 Since~$\delta^+(D)\geq \sum_{i=1}^\ell k_i$, there exists a $(v,u)$-dipath~$P_{v, u}$ in~$D$  of length exactly~$k_1$, for some vertex~$u \in V(D)$.
 Let $D' =   D - (P_{v, u} - u)$, let $C$ be the connected component of~$D'$ containing~$u$,
and let~$H$ be a \emph{sink} strong component of~$C$ (i.e. a strong component without arcs leaving it) that is reachable by a
 directed path in~$C$ starting at~$u$.
 We denote by~$P_{u,x}$ a $(u,x)$-dipath in~$C$ such that $V(P_{u,x}) \cap V(  H) = \{x\}$.

 Note that no vertex of~$H$ dominates a vertex in~$V(P_{u,x})\setminus\{x\}$
 since~$H$ is a sink strong component.
 Thus, $\delta^+(H) \geq \delta^+(D) - k_1 \geq k_2$.
 As a consequence, ~$H$ contains a directed cycle of length
 at least~$k_2$.
 Using this and the fact that~$H$ is strongly connected,
 we conclude that there exists a dipath~$P_{y,x}$ in~$H$
 from a vertex~$y \in V(H)\setminus\{x\}$ to~$x$ of length exactly~$k_2$.
 Let $G=H - (P_{y,x} - y)$.
 One may easily verify that $\delta^+(G)$ is at least~$\delta^+(D) - k_1 - k_2 \geq \sum_{i=3}^\ell k_i$.

 Let~$  Q_{v,y} =   P_{v, u}   P_{u,x}   P_{y,x}$.
 Note that~$  Q_{v,y}$ is a path~$P(k^\prime_1, k_2)$ starting at~$v$ with~$k^\prime_1 \geq k_1$.
 Therefore, the result follows immediately if~$\ell = 2$.
 Suppose now that~$\ell \geq 3$.
 By the induction hypothesis, $G$ contains a path~$W_y := P(k^\prime_3, \ldots, k^\prime_\ell)$ with initial vertex~$y$
 such that $k^\prime_i \geq k_i$ if~$i$ is odd, and $k^\prime_i = k_i$ otherwise.
 Therefore, $ Q_{v,y} W_y$ is the desired path~$P(k^\prime_1, k^\prime_2, \ldots, k^\prime_\ell)$
 with initial vertex~$v$.
\end{proof}

Since $\sum_{i=1}^\ell k_i = |V(P(k_1, k_2, \ldots, k_\ell))|-1$, and that the complete digraph $\vec{K}_k$ on $k$ vertices has minimum out-degree $k-1$ and contains no path on more than $k$ vertices, we obtain the following corollary.
\begin{corollary}
$\mad_{\delta^+}(P) = |V(P)| -1$ for every oriented path $P$.
\end{corollary}


\subsection{Subdivision of in-arborescences}\label{subsec:inarb}

The aim of this subsection is to prove that the in-arborescences are $\delta^+$-maderian.
We need some preliminary results. The first one is the vertex and directed version of the celebrated Menger's theorem \cite{Menger27}. (See also \cite{Goring00} for a short proof.)

\begin{theorem}[Menger's theorem]\label{theorem:menger}
 Let $D$ be a digraph, and let $S, T  \subseteq V(D)$.
 The maximum number of vertex-disjoint $(S,T)$-dipaths is equal to the minimum size of an $(S,T)$-vertex-cut.
\end{theorem}

\begin{lemma}\label{lemma:path_system}
  Let~$D$ be a digraph, let~$S\subseteq V(D)$ be a nonempty subset of vertices of in-degree~0 in~$D$,
  and let~$T \subseteq V(D)$ such that $T\cap S = \emptyset$.
  If $d^+(v) \geq \Delta^-(D)$ for all~$v \in V(D)\setminus T$,
  then there exist $|S|$ vertex-disjoint $(S,T)$-dipaths  in~$D$.
 \end{lemma}
 \begin{proof}

  Suppose to the contrary that there do not exist $|S|$ vertex-disjoint $(S,T)$-dipaths  in~$D$.
  By Menger's theorem,  there exists a set of vertices~$X \subseteq V(D)$ with cardinality~$|X| < |S|$
  which is an $(S,T)$-vertex-cut.
  Let~$C$ be the set of vertices in~$D - X$ that are reachable in~$D$ by a dipath with initial vertex in~$S\setminus X$.
  Set $k=|X \cap S|$. Observe that  $k < |S|$.

Let us count the number $a(C,X)$ of arcs with tail in~$C$ and head in~$X$.
Since the vertices in~$S$ have in-degree~0 and every vertex in~$C$ has out-degree at least~$\Delta^-(D)$,
$$a(C,X) \geq |C| \cdot \Delta^-(D) - [|C| - (|S| - k)]\cdot \Delta^-(D) = (|S| -k)\cdot \Delta^-(D).$$
 Moreover, $a(C,X)$ is at most the number of arcs with head in~$X$ which is at most~$(|X|-k) \cdot \Delta^-(D)$, because the vertices in $S\cap X$ have in-degree $0$.
Hence $(|S| -k)\cdot \Delta^-(D) \leq a(C,X) \leq (|X|-k) \cdot \Delta^-(D)$. This is a contradiction to $|X|<|S|$.
 \end{proof}

 Let~$k$ and~$\ell$ be positive integers.
 The {\it $\ell$-branching in-arborescence of depth~$k$}, denoted by $B(k,\ell)$, is defined by induction as follows.
 \begin{itemize}
 \item $B(0,\ell)$ is the in-arborescence with a single vertex, which is the root and the leaf of $B(0,\ell)$.
 \item $B(k,\ell)$ is obtained from $B(k-1,\ell)$ by taking each leaf of $B(k-1,\ell)$ in turn and adding $\ell$ new vertices dominating this leaf. The root of $B(k,\ell)$ is the one of $B(k-1,\ell)$, and the leaves of $B(k,\ell)$ are the newly added vertices, that is, those not in $V(B(k-1,\ell))$.
 \end{itemize}
 The number of vertices of $B(k,\ell)$ is denoted by~$b(k,\ell)$; so $b(k,\ell) = \sum_{i=0}^{k} \ell^i=\frac {1-\ell^{k+1}} {1-\ell}$.

Observe that every in-arborescence $T$ is a subdigraph of $B(k,\ell)$ with $k=\Delta^-(T)$ and
$\ell$ the maximum length of a dipath in $T$.
Therefore to prove that in-arboresences are $\delta^+$-maderian, it suffices to show that $B(k,\ell)$ is $\delta^+$-maderian for all $k$ and $\ell$.

 We define a recursive function~$f \colon \N \to \N$ as follows. For all positive integers~$k$ and~$\ell$
 such that~$\ell\geq 2$,
 $f(1, \ell) = \ell$ and, for~$k\geq 2$, we define
\[f(k, \ell)   = t(k,\ell) \cdot (\ell-1) \cdot k + t(k,\ell),\] where
 $t(k,\ell) := f(k-1, \: b(k-1,\ell)\cdot(\ell-1) +1)\cdot b(k-1, \ell)$.


If ${\cal D}$ is a family of digraphs, a {\it packing} of elements of ${\cal D}$ is the disjoint union of copies of elements of~${\cal D}$.

\begin{theorem}\label{inTrees}
 Let~$k$ and~$\ell\geq 2$ be positive integers, and let $D$ be a digraph with~$\delta^+(D) \geq f(k,\ell)$.
 Then~$D$ contains a subdivision of $B(k,\ell)$.
\end{theorem}
\begin{proof}
  We prove the result by induction on~$k$ and~$\ell$.
  If~$k=1$, then~$\delta^+(D) \geq \ell$.
  Thus, $\Delta^-(D) \cdot |V(D)| \geq \sum_{v \in V(D)} d^-(v) = \sum_{v \in V(D)} d^+(v) \geq \ell\cdot|V(D)|$. Hence
 there is a vertex with in-degree at least~$\ell$ in~$D$ and, consequently, the result follows when~$k=1$.
  Assume now~$k\geq 2$, and suppose that, for every positive integers~$k^\prime<k$ and~$\ell^\prime$, and every digraph~$H$
  with~$\delta^+(H) \geq f(k^\prime, \ell^\prime)$, $H$ contains a subdivision of
  the  $\ell^\prime$-branching in-arborescence of depth~$k^\prime$.

  Let~$\mathcal{F}$ be a packing of $\ell$-branching in-arborescences subdigraphs 
  of any non-zero depth in $D$ that covers the maximum number of vertices.
  We denote by~$U \subseteq V(D)$ the set of vertices not covered by~$\mathcal{F}$, that is,
  $U = V(D) \setminus \bigcup_{ A \in \mathcal{F}} V(  A)$.
  Let~$r_{  A}$ denote the root of the in-arborescence~$  A$, for each~$  A \in \mathcal{F}$,
  and let~$R=\{r_{  A} \in V(D) \colon   A \in \mathcal{F}\}$ be the set of the roots of the arborescences in~$\mathcal{F}$.

  We now construct the digraph~$H$ with vertex set~$R$ such that there exists an arc~$(r_{  A}, r_{  B})$ in~$H$ if and only if $r_{A}$ dominates some vertex of~$V(B)$ in~$D$.

\begin{claim}\label{claim:induction_step}
  If~$\delta^+(H) \geq t(k,\ell)/b(k-1, \ell)$,
  then $D$ contains a subdivision of  $B(k,\ell)$.
\end{claim}
\begin{subproof}
Let $p=b(k-1,\ell)\cdot(\ell-1) +1$. 
By the induction hypothesis,~$  H$ contains  a subdivision~$T$ of $B(k-1, p)$.
  Let~$R^\prime$ be the set of branching vertices of~$  T$, that is, $R' = \{ r \in T \colon d^-_{  T}(r) = p\}$.
  We assume that each in-arborescence of $\mathcal F$ has at most~$b(k-1,\ell)$ vertices, as any larger arborescence would yield the theorem.
  Thus, for each $r\in R^\prime$, there exists a vertex~$h_r$ in the in-arborescence rooted at~$r$
  such that~$h_r$ is dominated in~$D$ by $\ell$ vertices of~$V(T)$.
  Similarly, for each~$r \in V(T)$ with in-degree~1,   there exists a vertex~$h_r$ in the in-arborescence rooted at~$r$
  such that~$h_r$ is dominated in~$D$ by a vertex of~$V(T)$.
  Using these remarks, we next define a procedure to obtain a subdigraph of~$T$
  that is a subdivision of $B(k-1,\ell)$.

  For each~$r \in R^\prime$, we remove from~$  T$ all arcs with head~$r$ but exactly~$\ell$ arcs from vertices in~$V(T)$
  that dominate~$h_r$ in~$D$.
  We denote by~$  T^\prime$ the component of the subdigraph of~$T$ obtained by applying the above-described procedure 
  and that contains the root of~$  T$.
  One may easily verify that~$T^\prime$ is a subdivision of $B(k-1,\ell)$.
  Let~$P_r$ be the path from $h_r$ to~$r$ in the in-arborescence corresponding to~$r$, for every~$r \in V(T^\prime)$
  such that either~$r \in R^\prime$ or~$r$ has in-degree~1.

  Let~$Q$ be the in-arborescence
  obtained from~$T^\prime$  in the following way.
  For each~$r \in V(T^\prime)$ such that either~$r \in R^\prime$ or~$r$ has in-degree~1, we add~$h_r$ to~$T^\prime$,
  and we add an arc from every in-neighbour of~$r$ in~$T^\prime$ to~$h_r$.
  Additionally, we remove all arcs with head~$r$ in~$T^\prime$, and
  link~$h_r$ to~$r$ by using the dipath~$P_r$.
  Finally, for each~$r \in V(T^\prime)$ that is a leaf, we replace~$r$ by its corresponding in-arborescence belonging to~$\mathcal{F}$.

  By this construction, we have that~$Q$ is a subdigraph of~$D$ such that every internal vertex
  has either in-degree~$\ell$ or~1.
  Furthermore, it has depth at least~$k$.
  Therefore, by possibly pruning some levels of~$Q$, we obtain a subdivision of $B(k,\ell)$.
\end{subproof}

  Suppose now~$\delta^+(  H) < t(k,\ell)/b(k-1, \ell)$.
  Observe that, for every~$v \in R$ such that~$d^+_H(v) < t(k,\ell)/b(k-1, \ell)$,
  we have, in the digraph~$D$, that~$d^+_U(v) \geq t(k,\ell) \cdot (\ell-1) \cdot k$
  since~$\delta^+(D) \geq t(k,\ell) \cdot (\ell-1) \cdot k + t(k,\ell)$.
  We define \[X = \{v \in R \colon d^+_U(v) \geq t(k,\ell) \cdot (\ell-1) \cdot k\}.\]

  Let~$D^\prime$ be the digraph obtained from~$D[U \cup X]$ by removing all arcs with head in~$X$.
  From~$D^\prime$, we construct a digraph~$G$ by replacing every vertex~$v \in X$
  by $t(k,\ell)$ new vertices~$v_1, \ldots, v_{t(k,\ell)}$, and adding, for each~$i\in [t(k,\ell)]$,
  at least~$(\ell-1) \cdot k$ 
  arcs from~$v_i$ to~$N^+_{D^\prime}(v)$ in such a way
  that~$d^-_{D^\prime}(u) = d^-_G(u)$, for all~$u \in N^+_{D^\prime}(v)$.
  In other words, we ``redistribute'' the out-neighbours of~$v$ in~$D^\prime$ among
  its~$t(k,\ell)$ copies in~$G$ so that every copy has out-degree 
  at least~$(\ell-1) \cdot k$, and
  the in-degrees of vertices belonging to~$U$ are not changed.
  Let~$S \subseteq V(G)$ be the set of vertices that replaced those of $X$, that is,
  $S = \bigcup_{v \in X } \; \{v_1,\dots,v_{t(k,\ell)}\}$.
  Let~$T$ be the set of vertices in~$U$ that have large out-degree outside~$U$ in the digraph~$D$, more formally,
  $T = \{ v \in U \colon d^+_{V(D)\setminus U}(v) \geq t(k,\ell)+1 \}$.

  For every~$i \in [k-1]$, let~$\mathcal{F}_i = \{A \in \mathcal{F} \colon A \text{ has depth exactly } i\}$.
  Note that $\{\mathcal{F}_i\}_{i \in [k-1]}$ forms a partition of the packing~$\mathcal{F}$.
  Additionally, observe that, due to the maximality of~$\mathcal{F}$,
  every vertex in~$U$ is dominated by at most~$\ell-1$ vertices belonging to~$U$, and by at most~$\ell-1$ 
  roots of in-arborescences in~$\mathcal{F}_i$, for each~$i\in [k-1]$.
  Thus, the in-degree in~$G$ of every vertex belonging to~$U$ is at most~$(\ell-1) + (\ell-1) \cdot (k-1) = (\ell-1) \cdot k$.
Therefore, we have~$\Delta^-(G) \leq (\ell-1) \cdot k$.
Moreover, since~$\delta^+(D)\geq t(k,\ell) \cdot (\ell-1) \cdot k + t(k,\ell)$,
we have~$d^+_G(v) \geq t(k,\ell) \cdot (\ell-1) \cdot k$ for every~$v \in U\setminus T$.
Hence, $d^+_G(v) \geq (\ell-1) \cdot k$, for every~$v \in V(G) \setminus T$.
  By Lemma~\ref{lemma:path_system}, there exists a set~$\mathcal{P}$ of~$|S|$ vertex-disjoint paths from~$S$ to~$T$  in~$G$.

  Note that, in~$D$, every vertex belonging to~$T$ has at least~$t(k,\ell)+1$ out-neighbours in~$V(D)\setminus U$.
  Therefore one can greedily extend each path of ${\cal P}$ with an out-neighbour of its terminal vertex in $V(D)\setminus U$ in order to obtain a set~$\mathcal{P}^\prime$ of $|S|$ vertex-disjoint paths from $S$ to $V(D)\setminus U$ such that for any $v\in X$ all the paths in ${\mathcal P}'$ with initial vertex $v$ have distinct terminal vertices (and different from $v$).
  


  We now construct the digraph~$M$ on the vertex set~$R$ where there exists an
  arc from $v=r_A$ to $r_B$ in~$M$ whenever
  \begin{itemize}
  \item[] either $r_A$ dominates some vertex of $V(B)$ in $D$,
  \item[] or
  there is a dipath from some $v_i$ to $V(B)$ in ${\mathcal P}'$.
 \end{itemize}   
  Since, for each~$v \in X$, all vertices in~$\{v_i\}_{i \in [t(k,\ell)]}$ are the initial vertices of vertex-disjoint dipaths in~$\mathcal{P}^\prime$,
  we obtain~$\delta^+(M) \geq t(k,\ell)/b(k-1, \ell)$.
  Therefore, the result follows by Claim~\ref{claim:induction_step} with~$M$ playing the role of~$H$.
\end{proof}


\subsection{Cycles with two blocks}

We denote by $C(k_1, k_2)$ the digraph which is the union of two internally disjoint dipaths, 
one of length $k_1$ and one of length $k_2$ with the same initial vertex and same terminal vertex.
$C(k_1,k_2)$ may also be seen as an oriented cycles with two blocks, one of length $k_1$ and one of length $k_2$. Recall that the {\it blocks} of an oriented cycle are its maximal directed subpaths.

\begin{theorem}\label{C(k,k)}
Let $D$ be a digraph with $\delta^+(D) \geq 2(k_1+k_2)-1$. Then $D$ contains a subdivision of $C(k_1,k_2)$.
\end{theorem}
\begin{proof}
Let us assume, without loss of generality, that $k_1 \geq k_2$. 
Let $\ell$ be a positive integer.
An {\it $(\ell,k_1,k_2)$-fork} is a digraph obtained from the union of three disjoint  dipaths $A=(a_0,a_1,\cdots,a_{\ell})$, $B^1=(b^1_1,\cdots,b^1_{k_1-1})$ and $B^2=(b^2_1,\cdots,b^2_{k_2-1})$ by adding the arcs $(a_{\ell},b^1_1)$ and $(a_{\ell},b^2_1)$.

Since a $(1,k_1,k_2)$-fork has $k_1+k_2+1$ vertices and $\delta^+(D) \geq k_1+k_2+1$, 
then $D$ contains a $(1,k_1,k_2)$-fork as a subdigraph.
Let $\ell \geq 1$ be the largest integer such that $D$ contains an $(\ell,k_1,k_2)$-fork as a subdigraph.
Let $F$ be such a fork. For convenience, we denote its subpaths and vertices by their labels in the above definition.

If there exist $i,j \in [\ell]\cup \{0\}$, where~$i\leq j$ (resp. $j \leq i $),
such that $a_i \in N^+(b^1_{k_1-1})$ and $a_j \in N^+(b^2_{k_2-1})$, 
then the union of the dipaths $(a_{\ell},B^1,a_i,\cdots,a_j)$ and $(a_{\ell},B^2,a_j)$ (resp.  $(a_{\ell},B^1,a_i)$ and $(a_{\ell},B^2,a_j,\cdots,a_i)$)  is a subdivision of $C(k_1,k_2)$.

Suppose now that~$b^1_{k_1-1}$ has no out-neighbour in $\{a_0,\cdots,a_{\ell-1}\}$, that is,
$N^+(b^1_{k_1-1}) \cap (A \setminus \{a_{\ell}\}) = \emptyset$ (the case $N^+(b^2_{k_2-1}) \cap (A \setminus \{a_{\ell}\}) = \emptyset$ is similar). Since $|B^1\cup B^2 \cup \{a_{\ell}\}|=k_1+k_2-1$ and $\delta^+(D) \geq 2(k_1+k_2-1)+1$, $b^1_{k_1-1}$ has two distinct out-neighbours, say $c^1_1$ and $c^2_1$, not in $F$. 

Let $i_1 \geq 1$ be the largest integer such that there exist two disjoint dipaths $C^1$ and $C^2$ in $D-F$ with initial 
vertex $c^1_1$ and $c^2_1$, respectively, and length $i_1$ and $i_2=\min \{k_2,i_1\}$.
Set $C^1=(c^1_1,\cdots,c^1_{i_1})$ and $C^2=(c^2_1,\cdots,c^2_{i_2})$. 
By maximality of $\ell$, if $i_1\geq k_2$,  then $i_1 <k_1-1$. 
Otherwise, the union  of $A \cup B^1$, $C^1$, $C^2$, $(b^1_{k_1-1}, c^1_1)$ and
$(b^1_{k_1-1}, c^2_1)$ would contain an ($\ell+k_1-1,k_1,k_2)$-fork, contradicting the maximality of~$\ell$.

Suppose to the contrary that both $c^1_{i_1}$ and $c^2_{i_2}$ have no out-neighbour in $A \setminus  \{a_{\ell}\}$.
Since  $|V(B^1) \cup V(B^2) \cup \{a_{\ell}\}\cup V(C^1) \cup V(C^2)|=k_1+k_2+i_1+i_2-1<2(k_1+k_2)-2$ 
(since $i_1+i_2<k_1+k_2-1$) and $\delta^+(D)\geq 2(k_1+k_2-1)+1$, 
then there exist $c^1_{i_1+1}, c^2_{i_2+1} \in V(D -(F \cup C^1\cup C^2))$ 
such that $(c^1_{i_1},c^1_{i_1+1}), (c^2_{i_2},c^2_{i_2+1}) \in A(D)$ and $c^1_{i_1+1}\neq  c^2_{i_2+1}$. 
This contradicts the maximality of $i_1$.
Henceforth, we assume that  $c^1_{i_1}$ has an out-neighbour $a_j \in A \setminus \{a_{\ell}\}$ for some $0 \leq j < \ell$. 
The case  in which~$c^2_{i_2}$  has an out-neighbour in $A \setminus \{a_{\ell}\}$ is similar.

If $b^2_{k_2-1}$ has also an out-neighbour $a_m \in A \setminus \{a_{\ell}\}$, then the union of the dipaths $(a_{\ell},B^1,C^1,a_j,\dots,a_m)$ and $(a_{\ell},B^2,a_m)$ (if $m \geq j$) or of the dipaths $(a_{\ell},B^1,C^1,a_j)$ and $(a_{\ell},B^2,a_m,\dots,a_j)$ (if $m<j$) is a subdivision of $C(k_1,k_2)$.

If $b^2_{k_2-1}$ has an out-neighbour $z \in V(C^1\cup C^2)$, say $z=c^1_h$ for some $h \leq i_1$
(the case in which $z \in V(C^2)$ is similar), 
then the union of the dipaths $(a_{\ell},B^2,c^1_h)$ and $(a_{\ell},B,c^1_1,\cdots,c^1_h)$ is a subdivision of $C(k_1,k_2)$.

So, we may assume that $b^2_{k_2-1}$ has no out-neighbour in $A \setminus \{a_{\ell}\} \cup C^1\cup C^2$. 
Hence, $b^2_{k_2-1}$ has two distinct out-neighbours, say $c^3_1$ and $c^4_1$, not in $F \cup C^1\cup C^2$. 
Let $i_3 \geq 1$ be the largest integer such that
there exist two disjoint dipaths $C^3$ and $C^4$ in $D-(F\cup C^1 \cup C^2)$ with initial vertex $c^3_1$ and $c^4_1$, respectively,
and length $i_3$ and $i_4=\min \{k_2,i_3\}$.  
By the maximality of $\ell$, if $i_4 \geq k_2$ then $i_3 <k_1-1$ since otherwise the union of
$A \cup B^2$, $C^3$, $C^3$, $(b^2_{k_2-1}, c^3_1)$ and
$(b^2_{k_2-1}, c^4_1)$ would contain an ($\ell+k_1-1,k_1,k_2)$-fork, contradicting the maximality of $\ell$.

 For sake of contradiction, assume that both $c^3_{i_3}$ and $c^4_{i_4}$ have no out-neighbour in $(A \setminus  \{a_{\ell}\}) \cup C^1\cup C^2$. Because  $|V(B^1) \cup V(B^2) \cup V(C^3) \cup V(C^4) \cup  \{a_{\ell}\}|=k_1+k_2+i_3+i_4-1<2(k_1+k_2)-2$ (since $i_3+i_4 < k_1+ k_2-1$) and $\delta^+(D)\geq 2(k_1+k_2-1)+1$, then there exist distinct vertices $c^3_{i_3+1}, c^4_{i_4+1} \in V(D -(F \cup C^1 \cup C^2 \cup  C^3 \cup C^4))$ (if $i_3 \geq k_2$, we only define $c_{i_3+1}$) such that $(c^3_{i_3},c^3_{i_3+1}), (c^4_{i_4},c^4_{i_4+1}) \in A(D)$. This contradicts the maximality of $i_3$.

So one of $c^3_{i_3}, c^4_{i_4}$ has an out-neighbour in $(A \setminus  \{a_{\ell}\}) \cup C^1\cup C^2$.
We assume that it is $c^3_{i_3}$; the case when it is $c^4_{i_4}$ is similar.

If  $c^3_{i_3}$ has an out-neighbour $a_q \in A \setminus  \{a_{\ell}\}$ (for some $q < \ell$), then  the union of either the dipaths $(a_{\ell},B^1,C^1,a_j,\ldots,a_q)$ and $(a_{\ell},B^2,C^3,a_q)$ (if $q \geq j$), or the dipaths $(a_{\ell},B^1,C^1,a_j)$ and $(a_{\ell},B^2,C^3,a_q,\cdots,a_j)$ (if $q<j$), is a subdivision of $C(k_1,k_2)$.

If  $c^3_{i_3}$ has an out-neighbour $c^1_h \in V(C^1)$  for some $1\leq h \leq i_1$, then the union of the dipaths $(a_{\ell},B^1,c^1_1,\cdots,c^1_h)$ and $(a_{\ell},B^2,C^3,c^1_h)$ is a subdivision of $C(k_1,k_2)$.
Similarly, we find a  subdivision of $C(k_1,k_2)$ if $c^3_{i_3}$ has an out-neighbour in $C^2$.
\end{proof}

Theorem~\ref{C(k,k)} shows that an oriented cycle with two blocks $C$ is maderian and
\mbox{$\mad_{\delta^+}(C)\leq 2|V(C)|-1$}. A natural question is to ask whether this upper bound is tight or not.

\begin{problem}
What it the value of $\mad_{\delta^+}(C(k_1,k_2))$ ?
\end{problem}

\begin{proposition}
For any positive integer $k$, $\mad_{\delta^+}(C(k,1))=\mad_{\delta^0}(C(k,1))=k$.
\end{proposition}
\begin{proof}
The complete digraph on $k$ vertices has minimum in- and out-degree $k-1$, and it trivially contains no subdivision of $C(k,1)$ because it has less vertices than $C(k,1)$. Hence $\mad_{\delta^+}(C(k,1))\geq \mad_{\delta^0}(C(k,1))\geq k$.

Consider now a digraph $D$ with $\delta^+(D)\geq k$.
Let $P$ be a longest dipath in $D$ and let $u$ be its terminal vertex.
Necessarily $N^+(u)\subseteq V(P)$. Let $v$ (resp. $w$) be the first (resp. last) vertex of $N^+(u)$ along~$P$.
The path $P[v,w]$ contains all vertices of $N^+(u)$, so it has length at least $k-1$. Hence the union of $(u,v)\cup P[v,w]$ and
$(u,w)$ is a subdivision of $C(k,1)$.
\end{proof}

A step further towards Conjecture~\ref{mader} would be to prove that every oriented cycle is $\delta^+$-maderian. We even conjecture that  for every oriented cycle $C$ $\mad_{\delta^+}(C)\leq 2|V(C)|-1$.

\begin{conjecture}
Let $D$ be a digraph with $\delta^+(D) \geq 2k-1$. Then $D$ contains a subdivision of any oriented cycle of order $k$
\end{conjecture}

\subsection{Three dipaths between two vertices}

A slight adaptation of the proof of Theorem~\ref{C(k,k)} leads to a stronger result.
Let $k_1,k_2,k_3$ be positive integers. Let $P(k_1,k_2;k_3)$ be the digraph formed by three internally disjoint paths between two vertices $x,y$, two $(x,y)$-dipaths, one of size at least $k_1$, the other of size at least $k_2$, and one $(y,x)$-dipath of size at least $k_3$.
When we want to insist on the vertices $x$ and $y$, we denote it by $P_{xy}(k_1,k_2;k_3)$.

\begin{theorem}\label{triplePath}
Let $k_1,k_2,k_3$ be  positive integers with $k_1 \ge k_2$.
Let $D$ be a digraph with $\delta^+(D) \geq 3k_1+2k_2+k_3-5$. Then $D$ contains  $P(k_1,k_2;k_3)$.
\end{theorem}

\begin{proof}
Let $\ell$ be an integer.
 An {\it $(\ell,k_3;k_1,k_2)$-fork} is a digraph obtained from the union of four disjoint  directed paths $P=(p_1, \dots, p_{\ell})$, $A=(a_1,\cdots,a_{k_3-1})$, $B^1=(b^1_1,\cdots,b^1_{k_1-1})$ and $B^2=(b^2_1,\cdots,b^2_{k_2-1})$ by adding the arcs $(p_{\ell},a_1)$, $(a_{k_3-1},b^1_1)$ and $(a_{k_3-1},b^2_1)$.

Since a $(1,k_3;k_1,k_2)$-fork has $k_1+k_2+k_3-2$ vertices and $\delta^+(D) \geq k_1+k_2+k_3-2$, then $D$ contains a $(1,k_3;k_1,k_2)$-fork as a subdigraph.
So, let $\ell \geq 1$ be the largest integer such that $D$ contains an $(\ell,k_3;k_1,k_2)$-fork as a subdigraph. Let $F$ be such a fork.
For convenience, we denote its subpaths and vertices by their labels in the above definition.

If there exist $ i, j \in [\ell]$, with $i \leq j$,
such that $p_i \in N^+(b^1_{k_1-1})$ and $p_j \in N^+(b^2_{k_2-1})$ or $p_i \in N^+(b^2_{k_2-1})$ and $p_j \in N^+(b^1_{k_1-1})$, then $F$ contains a $P_{a_{k_3-1}p_j}(k_1,k_2;k_3)$.

So, let us assume that $b^1_{k_1-1}$ has no out-neighbour in $P$ (the case where $b^2_{k_2-1}$ has no out-neighbour in $P$ is similar). 
Since $|A \cup B^1\cup B^2|=k_1+k_2+k_3-3$ and $\delta^+(D) \geq k_1+k_2+k_3-1$, $b^1_{k_1-1}$ has two distinct out-neighbours, say $c^1_1$ and $c^2_1$, not in $F$.

Let $i_1 \geq 1$ be the largest integer such that there exist two disjoint directed paths $C^1$ and $C^2$ in $D-F$ with initial vertex $c^1_1$ and $c^2_1$ respectively and length $i_1$ and $i_2=\min\{k_2-1,i_1\}$.
If $i_1 \ge k_1-1$, then $i_2 \ge k_2-1$, and  thus  $P \cup A \cup B^1 \cup C_1 \cup C_2$ would contain a fork that contradicts the maximality of $\ell$.
Hence we may assume that $i_1 \le  k_1-2$ (and in particular $|V(C_1) \cup V(C_2)| \le k_1+k_2-3$).

For sake of contradiction, assume that both $c^1_{i_1}$ and $c^2_{i_2}$ have no out-neighbour in $P$.
Since  $|V(A) \cup V(B^1) \cup V(B^2)  \cup V(C^1) \cup V(C^2)|\le  2k_1+2k_2+ k_3-6< \delta^+(D)-2$,
then there exist $c^1_{i_1+1}, c^2_{i_2+1} \in V(D -(F \cup C^1\cup C^2))$ such that $(c^1_{i_1},c^1_{i_1+1}), (c^2_{i_2},c^2_{i_2+1}) \in A(D)$ and $c^1_{i_1+1}\neq  c^2_{i_2+1}$. This contradicts the maximality of $i_1$.
Henceforth, we assume that  $c^1_{i_1}$ has an out-neighbour $p_i \in P$ 
(the case  in which $c^2_{i_2}$  has an out-neighbour in $P$ is similar).

If $b^2_{k_2-1}$ has also an out-neighbour $p_j \in P$,
then $F \cup C_1$ contains a $P_{a_{k_3}p_j}(k_1,k_2;k_3)$ if $i\le j$, and a  $P_{a_{k_3}p_i}(k_1,k_2;k_3)$ if $j \le i$.

So, we may assume that $b^2_{k_2-1}$ has no out-neighbour in $P$.
Hence, $b^2_{k_2-1}$ has two distinct out-neighbours, say $c^3_1$ and $c^4_1$, not in $F \cup C^1$.
Let $i_3 \geq 1$ be the largest integer such that
there exist two disjoint dipaths $C^3$ and $C^4$ in $D-(F\cup C^1 )$ with initial vertex $c^3_1$ and $c^4_1$ respectively and length $i_3$ and $i_4=\min \{k_2-1,i_3\}$.
If $i_3 \ge k_1$, then $i_4 \ge k_2-1$ and thus $P \cup A \cup B^2 \cup C_3 \cup C_4$ contains a fork that contradicts the maximality of $F$.
Thus, we may assume that $i_3 \le k_1-2$. In particular $|V(C_3) \cup V(C_4)| \le k_1+k_2-3$.

 Suppose to the contrary that both $c^3_{i_3}$ and $c^4_{i_4}$ have no out-neighbour in $P$, where $c^3_{i_3}$ and $c^4_{i_4}$ are the last vertices of $C^3$ and $C^4$.
  Note that  $|V(A) \cup V(B^1) \cup V(B^2) \cup V(C^1) \cup V(C^3) \cup V(C^4)|\le 3k_1+2k_2+k_3-7 \le \delta^+(D)-2$.
 Hence, there exist distinct vertices $c^3_{i_3+1}, c^4_{i_4+1} \in V(D -(F \cup C^1  \cup  C^3 \cup C^4))$  such that $(c^3_{i_3},c^3_{i_3+1}), (c^4_{i_4},c^4_{i_4+1}) \in A(D)$.
This contradicts the maximality of $i_3$.

Therefore, one of $c^3_{i_3}, c^4_{i_4}$ has an out-neighbour in $p_j$ in $P$.
We assume that it is $c^3_{i_3}$; the case when it is $c^4_{i_4}$ is similar.
We conclude that $F \cup C_1 \cup C_3$ contains a  $P_{a_{k_3}p_j}(k_1,k_2;k_3)$ if $i<j$, and a $P_{a_{k_3}p_i}(k_1,k_2;k_3)$ if $j<i$.
\end{proof}

\section{Subdivisions in digraphs with large dichromatic number}\label{sec:dic}

Recall that a \emph{$k$-dicolouring} is a $k$-partition $\{V_1, \dots , V_k\}$ of $V(D)$ such that~$D\langle V_i \rangle$ is acyclic 
for every $i \in [k]$, and that the {\it dichromatic number} of $D$ is the minimum $k$ such that $D$ admits a $k$-dicolouring.
In this section, we first prove that every digraph is $\dic$-maderian.
We need some preliminaries. The first one is an easy lemma, whose proof is left to the reader.

\begin{lemma}\label{lem:dic-strong}
The dichromatic number of a digraph is the maximum of the dichromatic numbers of its strong components.
\end{lemma}

Our proof is based on levelling. 
Forthwith, we  introduce the necessary definitions.
Given a digraph~$D$, the {\it distance} from a vertex $x$ to another $y$, denoted by $\dist_D(x,y)$ or simply  $\dist(x,y)$ when $D$ is clear from the context, is the minimum length of an $(x,y)$-dipath or $+\infty$ if no such dipath exists.
An {\it out-generator} in~$D$ is a vertex $u$ such that, for every $x\in V(D)$, there exists an $(u,x)$-dipath in~$D$.
Analogously, an \emph{in-generator} in~$D$ is a vertex $u$ such that, for every $x\in V(D)$, there exists an $(x,u)$-dipath in~$D$.
For simplicity, we call a vertex \emph{generator} if it is an in- or out-generator.
 Observe that every vertex in a strong digraph is an in- and out-generator.


Let $D$ be a digraph. 
Let $w,u$ be in- and out-generators of $D$, respectively.
We remark that $w$ and $u$ are not necessarily different.
For every nonnegative integer $i$,  the {\it $i$th out-level  from $u$} in $D$ is the set
$L^{u,+}_i=\{v \in V(D) \mid \dist_D(u,v)=i\}$, and the {\it $i$th in-level  from $w$} in $D$ is the set
$L^{w,-}_i=\{v \in V(D) \mid \dist_D(v,w)=i\}$.
Note that $\bigcup_{i} L^{u, +}_i =\bigcup_{i} L^{w, -}_i =V(D)$.

An {\it out-Breadth-First-Search Tree}  or 
{\it out-BFS-tree} $T^+$ with {\it root~$u$}, is a subdigraph of $D$ spanning $V(D)$ such that $T^+$ is an oriented tree and,
for every $v \in V(D)$, $dist_{T^+}(u,v) =  \dist_D(u,v)$.
Similarly, an {\it in-Breadth-First-Search Tree}  or 
{\it in-BFS-tree} $T^-$ with {\it root~$w$}, is a subdigraph of $D$ spanning $V(D)$ such that $T^-$ is an oriented tree and,
for every $v \in V(D)$, $dist_{T^-}(v,w) =  \dist_D(v,w)$.



It is well-known that if~$D$ has an out-generator, then there exists an out-BFS-tree rooted at this vertex.
Likewise, if~$D$ has an in-generator, then there exists an in-BFS-tree rooted at this generator.


Let $T$ denote an in- or out-BFS-tree rooted at~$u$.
For any vertex $x$ of $D$, there is a single $(u,x)$-dipath in $T$ if~$T$ is an out-BFS-tree, and a single $(x,u)$-dipath in $T$ if~$T$ is an in-BFS-tree.
The {\it ancestors} or \emph{successors} of $x$ in~$T$ are naturally defined.
If $y$ is an ancestor of $x$, we denote by $T[y,x]$ the $(y,x)$-dipath in $T$.
If $y$ is a successor of $x$, we denote by $T[x,y]$ the $(x,y)$-dipath in $T$.

\begin{lemma}\label{lem:dic-level}
Let $D$ be a strong digraph and let $T$ be an in- or out-BFS-tree in $D$.
There is a level $L$ such that $\dic(D\langle L \rangle) \geq \dic(D)/2$.
\end{lemma}
\begin{proof}
First, let us suppose, without loss of generality, that~$T$ is an out-BFS-tree in~$D$. 
The proof when~$T$ is an in-BFS-tree is analogous.

Let $D_1$ and $D_2$ be the subdigraphs of $D$ induced by the vertices of odd and even levels, respectively.
Since there is no arc from $L_i$ to $L_j$ for every $j\geq i+2$, the strong components of $D_1$ and $D_2$ are contained in the levels.
Hence, by Lemma~\ref{lem:dic-strong}, $\dic(D_1) = \max \{\dic(D\langle L_i \rangle) \mid i \text{ is odd} \}$ and $\dic(D_2) = \max \{\dic(D\langle L_i \rangle) \mid i \text{ is even}\}$.
Moreover, note that $V(D_1)\cup V(D_2) =V(D)$ because $D$ is strong. 
Therefore, $\dic(D) \leq \dic(D_1)+\dic(D_2) \leq 2\cdot \max  \{\dic(D\langle L_i \rangle) \mid i\in \mathbb{N}\}$.
\end{proof}

\begin{lemma}\label{lem:recur-dic}
Let $F$ be a digraph  and let $a=xy$ be an arc in $A(F)$.
If $F - a$ is $\dic$-maderian, then $F$ is $\dic$-maderian, and $\mad_{\dic}(F) \le 4\cdot \mad_{\dic}(F - a)-3$
\end{lemma}
%
\begin{proof}
Let $c=\mad_{\dic}(F-a)$ and let~$D$ be a digraph with $\dic(D)\geq 4c-3$.
We shall prove that $D$ contains a subdivision of $F$.

By Lemma~\ref{lem:dic-strong}, we may assume that $D$ is strong.
Let $u$ be a vertex in $D$ and $T_u$ an out-BFS-tree with root $u$.
By Lemma~\ref{lem:dic-level}, there is a level $L^u$ such that  $\dic(D\langle L^u\rangle) \geq 2c-1$. 
By Lemma~\ref{lem:dic-strong}, there is a strong component $C$ of $D\langle L^u\rangle$ such that
$\dic(C) = \dic(D\langle L^u\rangle) \geq 2c-1$.
Since $D$ is strong, there is a shortest $(v, u)$-dipath~$P$ in~$D$ such that~$V(P)\cap V(C) = \{v\}$. 
Let $T_v$ be an  in-BFS-tree in $C$ rooted at~$v$.
By Lemma~\ref{lem:dic-level}, there is a level $L^v$ of $T_v$ such that  $\dic(D\langle L^v\rangle) \geq c$.
Now since $\mad_{\dic}(F - a)=c$, $D\langle L^v\rangle$ contains a subdivision $S$ of $F- a$. 
With a slight abuse of notation, let us call $x$ and $y$ the vertices in~$S$ corresponding to the vertices $x$ and $y$ of $F$.
Now $T_v[x,v]\cup P \cup T_u[u,y]$ is a directed $(x,y)$-walk with no internal vertex in $L^v$.
Hence it contains an $(x,y)$-dipath $Q$ whose internal vertices are not in $S$.
Therefore, $S\cup Q$ is a subdivision of $F$ in~$D$.
\end{proof}

\begin{theorem}\label{thm:dic-maderian}
Every digraph $F$ is $\dic$-maderian.
More precisely, $\mad_{\dic}(F)\leq 4^m (n-1) + 1$, where $m=|A(F)|$ and $n=|V(F)|$.
\end{theorem}
\begin{proof}
We prove the result by induction on $m$.
If $m=0$, then $F$ is an empty digraph which is trivially $\dic$-maderian and $\mad_{\dic}(F)=n$.
If $m>0$, then consider an arc~$a \in A(F)$.
By Lemma~\ref{lem:recur-dic}, we obtain $\mad_{\dic}(F) \leq 4 \cdot  \mad_{\dic}(F - a) -3$.
By the induction hypothesis, $\mad_{\dic}(F - a)\leq 4^{m-1} (n-1) + 1$.
Therefore, $\mad_{\dic}(F)\leq 4^m (n-1) + 1$.
\end{proof}

Observe that Theorem~\ref{thm:dic-maderian} generalizes the consequence of Theorem~\ref{thm:mader-undirected}, stating that every graph with sufficiently large chromatic number contains a subdivision of $K_k$. In fact this statement corresponds to the case of symmetric digraphs of Theorem~\ref{thm:dic-maderian}.

\subsection{Better bounds on $\mad_{\dic}$}\label{better-dic}

The bound on $\mad_{\dic}$ given in Theorem~\ref{thm:dic-maderian} is not optimal.
The aim of this subsection is to find better upper bounds.

A digraph is {\it $k$-$\dic$-critical} if $\dic(D) =k$ and $\dic (D')<k$ for every proper subdigraph $D'$ of $D$.

\begin{proposition}
If $D$ is $k$-$\dic$-critical, then $\delta^0(D)\geq k-1$.
\end{proposition}
\begin{proof}
Let $v$ be a vertex of $D$.
Since $D$ is $k$-$\dic$-critical, $\dic(D-v) \leq  k-1$, so $D-v$ admits a $(k-1)$-dicolouring $\{V_1, \dots , V_{k-1}\}$.
Thus, for each $ i \in [k-1]$, $D\langle V_i \cup \{v\}\rangle$  has a directed cycle that contains~$v$.
Therefore, $v$ has an in-neighbour and an out-neighbour in each $V_i$.
\end{proof}

\begin{corollary}
$\mad_{\dic}(F) \leq \mad_{\delta^0}(F) +1$ for all digraph $F$.
\end{corollary}

\begin{corollary}\label{cor:dic-forest}
$\mad_{\dic}(F) =|V(F)|$ for all oriented forest $F$.
\end{corollary}

Let us denote by $\cc(F)$ the number of connected components of $F$, that are the connected components of the underlying graph.

\begin{corollary}\label{cor:dic-Kn}
For every digraph $F$, we have
$\mad_{\dic}(F) \leq 4^{m-n+\cc(F)}(n-1)+1$, where $m=|A(F)|$ and $n=|V(F)|$.
\end{corollary}
\begin{proof}
The proof is identical to the one of Theorem~\ref{thm:dic-maderian}, but instead of starting the induction with empty digraphs, we start it with a forest that is the union of spanning trees of the connected components. 
\end{proof}

Corollary~\ref{cor:dic-Kn} implies that $\mad_{\dic}(\vec{K}_n) \leq 4^{n(n-2)+1}(n-1)+1$.
On the other hand, we have  $\mad_{\dic}(\vec{K}_n) \geq \Omega(\frac{ n^2}{\log n})$. 
Indeed,  consider a tournament $T$ on $p$ vertices with a subdivision $S$ of $\vec{K}_n$. 
For every two distinct vertices $u,v$ of $\vec{K}_n$, at least one of the arcs $(u,v)$, $(v,u)$ is subdivided in $S$. 
Hence, $S$ has at least $n + \binom{n}{2}=\binom{n+1}{2}$ vertices, so $p\geq \binom{n+1}{2}$.  
Erd\H{os} and Moser~\cite{erdos64} proved that for every integer $p$, 
there exists a tournament $T_p$ on $p$ vertices  with no transitive tournament of order $2\log p +1$. 
Thus $\dic(T_p) \geq \frac{p}{2\log p}$.
Now set $p = \binom{n+1}{2}-1$. The tournament $T_p$ contains no subdivision of $K_n$ and $\dic(T_p) \geq \frac{p}{2\log p}$.
Hence $\mad_{\dic}(\vec{K}_n) \geq  \frac{p}{2\log p} + 1 \geq \Omega(\frac{ n^2}{\log n})$.

\medskip

A {\it $k$-source} in a digraph is a vertex $x$ with in-degree $0$ and out-degree at most $k$; a {\it $k$-sink} in a digraph is a vertex $x$ with out-degree $0$ and in-degree at most $k$.
A digraph is {\it $k$-reducible} if it can be reduced to the empty digraph  by repeated deletion of $k$-sources or
$k$-sinks. For instance, the $1$-reducible digraphs are the oriented forests.

\begin{lemma}
Let $F$ be a digraph having a $2$-source $x$.
Then $\mad_{\dic}(F) \leq 2\mad_{\dic}(F-x) -1$.
\end{lemma}
\begin{proof}
Suppose that $F-x$ is $(\dic \geq c)$-maderian.
We shall prove that $F$ is $(\dic \geq 2c-1)$-maderian.

Let $D$ be a digraph with $\dic(D)\geq 2c-1$.
By Lemma~\ref{lem:dic-strong}, we may assume that $D$ is  strong.
Let $u$ be a vertex in $D$, and let $T$ be a BFS-tree with root $u$.
By Lemma~\ref{lem:dic-level}, there is a level such that $\dic (D\langle L\rangle) \geq c$.
Consequently, $D\langle L \rangle$ contains a subdivision $S$ of $F-x$.
Let $y_1$ and $y_2$ be the vertices in~$S$ corresponding to the two out-neighours of $x$ in $F$.
Let $v$ be the least common ancestor of $y_1$ and $y_2$ and, for~$i \in \{1,2\}$, let $P_i$ be the $(v,y_i)$-dipath in $T$.
Therefore, we conclude that the digraph $S\cup P_1 \cup P_2$ is a subdivision of $F$ in~$D$.
\end{proof}

\begin{corollary}
The following statements hold.
\begin{enumerate}[(a)]
\item\label{stat:reducible} $\mad_{\dic} (F) \leq 2^{|V(F)|-2}+1$ for every $2$-reducible digraph $F$ or order at least $2$.
\item\label{stat:oriented_cycle} $\mad_{\dic} (C) \leq 2\cdot |V(C)| -3$ for every oriented cycle $C$ of order at least $3$.
\end{enumerate}
\end{corollary}
\begin{proof}
Statement~\eqref{stat:reducible} follows by induction on~$|V(F)|$. 
Observe that the result trivially holds when $|V(F)|=2$.

To prove statement~\eqref{stat:oriented_cycle}, consider the following two complementary cases.
If  $C$ is directed, say $C=\vec{C}_k$, then $\mad_{\dic}(C)\leq \mad_{\delta^0}(C)+1\leq \mad_{\delta^+}(C)\leq k\leq 2k-3$. 
If $C$ is not directed, then it contains a $2$-source $x$. 
Hence $C-x$ is an oriented path, and, by Corollary~\ref{cor:dic-forest}, 
it follows that $\mad_{\dic}(C-x)=|C-x| = |C|-1$.
\end{proof}

\begin{conjecture}
 $\mad_{\dic} (C) \leq |C|$ for every oriented cycle $C$.
\end{conjecture}

\bibliographystyle{plain}
\bibliography{maderian,bibliography}

\begin{thebibliography}{10}

\bibitem{AHL+13}
Louigi Addario-Berry, Fr\'ed\'eric Havet, Cl\'audia~Linhares Sales, Bruce Reed,
  and St\'ephan Thomass\'e.
\newblock Oriented trees in digraphs.
\newblock {\em Discrete Mathematics}, 313(8):967 -- 974, 2013.

\bibitem{AKKW16}
S.~{Akhoondian Amiri}, K.-I. {Kawarabayashi}, S.~{Kreutzer}, and P.~{Wollan}.
\newblock {The Erdos-Posa Property for Directed Graphs}.
\newblock {\em ArXiv e-prints}, March 2016.

\bibitem{Alon96}
Noga Alon.
\newblock Disjoint directed cycles.
\newblock {\em Journal of Combinatorial Theory, Series B}, 68(2):167 -- 178,
  1996.

\bibitem{BeTh81}
J.~C. Bermond and C.~Thomassen.
\newblock Cycles in digraph -- a survey.
\newblock {\em Journal of Graph Theory}, 5(1):1--43, 1981.

\bibitem{BoTh98}
B.~Bollob\'as and A.~Thomason.
\newblock Proof of a conjecture of {M}ader, {E}rd\"{o}s and {H}ajnal on
  topological complete subgraphs.
\newblock {\em European Journal of Combinatorics}, 19(8):883 -- 887, 1998.

\bibitem{BrDo96}
Stephan Brandt and Edward Dobson.
\newblock The {E}rd\"{o}s-{S}\'os conjecture for graphs of girth 5.
\newblock {\em Discrete Mathematics}, 150(1):411 -- 414, 1996.

\bibitem{Bur80}
Stefan~A. Burr.
\newblock Subtrees of directed graphs and hypergraphs.
\newblock In {\em Proceedings of the Eleventh Southeastern Conference on
  Combinatorics, Graph Theory and Computing (Florida Atlantic Univ., Boca
  Raton, Fla., 1980), Vol. I}, volume~28, pages 227--239, 1980.

\bibitem{Bur82}
Stefan~A. Burr.
\newblock Antidirected subtrees of directed graphs.
\newblock {\em Canad. Math. Bull.}, 25(1):119--120, 1982.

\bibitem{CHLN}
Nathann Cohen, Fr{\'e}d{\'e}ric Havet, William Lochet, and Nicolas Nisse.
\newblock {Subdivisions of oriented cycles in digraphs with large chromatic
  number}.
\newblock Research Report RR-8865, {LRI - CNRS, University Paris-Sud ; LIP -
  ENS Lyon ; INRIA Sophia Antipolis - I3S}, February 2016.

\bibitem{DMMS12}
Matt DeVos, Jessica McDonald, Bojan Mohar, and Diego Scheide.
\newblock Immersing complete digraphs.
\newblock {\em European Journal of Combinatorics}, 33(6):1294 -- 1302, 2012.

\bibitem{Dirac60}
Gabriel~Andrew Dirac.
\newblock In abstrakten graphen vorhandene vollst\"andige 4-graphen und ihre
  unterteilungen.
\newblock {\em Mathematische Nachrichten}, 22(1-2):61--85, 1960.

\bibitem{Erd65}
P.~Erd\H{o}s.
\newblock Some problems in graph theory.
\newblock In {\em Theory of Graphs and Its Applications}, pages 29--36.
  Academic Press, New York, 1965.

\bibitem{erdos64}
P.~Erd{\H{o}}s and L.~Moser.
\newblock On the representation of directed graphs as unions of orderings.
\newblock {\em Magyar Tud. Akad. Mat. Kutat\'o Int. K\"ozl.}, 9:125--132, 1964.

\bibitem{Goring00}
F.~G{\"o}ring.
\newblock Short proof of {M}enger's theorem.
\newblock {\em Discrete Mathematics}, 219(1–3):295--296, 2000.

\bibitem{Haxell2001}
P.~E. Haxell.
\newblock Tree embeddings.
\newblock {\em Journal of Graph Theory}, 36(3):121--130, 2001.

\bibitem{KaKr15}
Ken-ichi Kawarabayashi and Stephan Kreutzer.
\newblock The directed grid theorem.
\newblock In {\em Proceedings of the Forty-Seventh Annual ACM on Symposium on
  Theory of Computing}, STOC '15, pages 655--664, New York, NY, USA, 2015. ACM.

\bibitem{KoSz96}
J\'anos Koml\'os and Endre Szemer\'edi.
\newblock Topological cliques in graphs ii.
\newblock {\em Combinatorics, Probability and Computing}, 5:79--90, 3 1996.

\bibitem{Lovasz75}
L.~Lov\'asz.
\newblock Problem 2.
\newblock In Miroslav Fiedler, editor, {\em Recent advances in graph theory,
  Proceedings of the {S}econd {C}zechoslovak {S}ymposium held in {P}rague,
  {J}une 1974}. Academia, Prague, 1975.

\bibitem{Mad67}
W.~Mader.
\newblock Homomorphieeigenschaften und mittlere kantendichte von graphen.
\newblock {\em Mathematische Annalen}, 174:265--268, 1967.

\bibitem{Mader85}
W.~Mader.
\newblock Degree and local connectivity in digraphs.
\newblock {\em Combinatorica}, 5(2):161--165, 1985.

\bibitem{Mader95}
Wolfgang Mader.
\newblock Existence of vertices of local connectivity \emph{k} in digraphs of
  large outdegree.
\newblock {\em Combinatorica}, 15(4):533--539, 1995.

\bibitem{Menger27}
Karl Menger.
\newblock Zur allgemeinen kurventheorie.
\newblock {\em Fundamenta Mathematicae}, 10(1):96--115, 1927.

\bibitem{RRST96}
Bruce Reed, Neil Robertson, Paul Seymour, and Robin Thomas.
\newblock Packing directed circuits.
\newblock {\em Combinatorica}, 16(4):535--554, 1996.

\bibitem{SaWo97}
Jean-Fran\c{c}ois Sacl\'e and Mariusz Wo\'zniak.
\newblock The {E}rd\"{o}s-{S}\'os conjecture for graphs without $c_4$.
\newblock {\em Journal of Combinatorial Theory, Series B}, 70(2):367 -- 372,
  1997.

\bibitem{SeTh87}
Paul Seymour and Carsten Thomassen.
\newblock Characterization of even directed graphs.
\newblock {\em Journal of Combinatorial Theory, Series B}, 42(1):36 -- 45,
  1987.

\bibitem{Thomassen74}
Carsten Thomassen.
\newblock Some homeomorphism properties of graphs.
\newblock {\em Mathematische Nachrichten}, 64(1):119--133, 1974.

\bibitem{Thomassen86}
Carsten Thomassen.
\newblock Sign-nonsingular matrices and even cycles in directed graphs.
\newblock {\em Linear Algebra and its Applications}, 75:27 -- 41, 1986.

\bibitem{Thomassen89}
Carsten Thomassen.
\newblock Configurations in graphs of large minimum degree, connectivity, or
  chromatic number.
\newblock {\em Annals of the New York Academy of Sciences}, 555(1):402--412,
  1989.

\bibitem{Thomassen96}
Carsten Thomassen.
\newblock K5-subdivisions in graphs.
\newblock {\em Combinatorics, Probability and Computing}, 5:179--189, 6 1996.

\end{thebibliography}

\end{document}